\documentclass[10pt,a4paper]{article}
\usepackage[final]{optional}
\usepackage[british,american]{babel}

\usepackage{latexsym}
\usepackage{amsxtra}
\usepackage{mathrsfs}

\usepackage{amsfonts, amsmath, wasysym}

\usepackage{amssymb,amsthm,
paralist
}

\usepackage{
latexsym,
nicefrac,
}

\usepackage[usenames]{color}

\usepackage{url}

\definecolor{darkgreen}{rgb}{0,0.5,0}
\definecolor{darkred}{rgb}{0.7,0,0}
\usepackage[colorlinks,
citecolor=darkgreen, linkcolor=darkred
]{hyperref}
\usepackage{esint}
\usepackage{bibgerm}
\usepackage[normalem]{ulem}

\theoremstyle{plain}
\newtheorem{lemma}{Lemma}[section]
\newtheorem{thm}[lemma]{Theorem}
\newtheorem{prop}[lemma]{Proposition}

\theoremstyle{definition}
\newtheorem{defn}[lemma]{Definition}

\newtheorem{rmk}[lemma]{Remark}

\def\blbox{\quad \vrule height7.5pt width4.17pt depth0pt}
\newcommand{\cmt}[1]{\opt{draft}{\textcolor[rgb]{0.5,0,0}{
$\LHD$ #1 $\RHD$\marginpar{\blbox}}}}

\numberwithin{equation}{section}

\newcommand{\m}{\ensuremath{{\cal M}}}
\newcommand{\n}{\ensuremath{{\cal N}}}

\newcommand{\cc}{\ensuremath{{\cal C}}}

\newcommand{\pl}[2]{{\frac{\partial #1}{\partial #2}}}

\newcommand{\ga}{\gamma}

\newcommand{\de}{\delta}

\newcommand{\si}{\sigma}

\newcommand{\Si}{\Sigma}
\renewcommand{\th}{\theta}
\newcommand{\Th}{\Theta}

\newcommand{\ep}{\varepsilon}

\newcommand{\R}{\ensuremath{{\mathbb R}}}
\newcommand{\N}{\ensuremath{{\mathbb N}}}

\newcommand{\Z}{\ensuremath{{\mathbb Z}}}
\newcommand{\C}{\ensuremath{{\mathbb C}}}

\newcommand{\downto}{\downarrow}

\DeclareMathOperator{\inj}{inj}

\newcommand{\norm}[1]{\Vert#1\Vert}

\def\blbox{\quad \vrule height7.5pt width4.17pt depth0pt}

\def\bee{\begin{eqnarray}}
\def\beee{\begin{eqnarray*}}
\def\eee{\end{eqnarray}}
\def\eeee{\end{eqnarray*}}
\def\nn{\nonumber}

\newcommand{\beq}{\begin{equation}}
\newcommand{\eeq}{\end{equation}}
\newcommand{\beqa}{\begin{equation}\begin{aligned}}
\newcommand{\eeqa}{\end{aligned}\end{equation}}
\newcommand{\beqas}{\begin{equation*}\begin{aligned}}
\newcommand{\eeqas}{\end{aligned}\end{equation*}}
\newcommand{\brmk}{\begin{rmk}}
\newcommand{\ermk}{\end{rmk}}
\newcommand{\partref}[1]{\hbox{(\csname @roman\endcsname{\ref{#1}})}}

\newcommand*\tr{\mathop{\mathrm{tr}}\nolimits}

\newcommand*\arsinh{\mathop{\mathrm{arsinh}}\nolimits}

\newcommand{\M}{\ensuremath{{\mathcal M}}_{-1}}

\newcommand{\abs}[1]{\left\vert#1\right\vert}

\newcommand{\eps}{\varepsilon}
\newcommand{\na}{\nabla}
\newcommand{\Hol}{{\cal H}} 

\newcommand{\Zn}{\Z\setminus\{0\}}  
\newcommand{\drangle}{\rangle\rangle}
\newcommand{\dlangle}{\langle\langle}
\newcommand{\Col}{\ensuremath{{\cal C}}}
\newcommand{\thin}{\text{-thin}}
\newcommand{\thick}{\text{-thick}}

\title{{\sc
asymptotics of the \\ Teichm\"uller harmonic map flow
}
\\
\cmt{DRAFT with comments}
}
\author{Melanie Rupflin, Peter M. Topping and Miaomiao Zhu}
\date{\today}

\begin{document}
\maketitle

\begin{abstract}
The Teichm\"uller harmonic map flow, introduced in \cite{R-T}, evolves both a map from a closed Riemann surface to an arbitrary compact Riemannian manifold, and a constant curvature metric on the domain, in order to reduce its harmonic map energy as quickly as possible.
In this paper, we develop the geometric analysis of holomorphic quadratic differentials in order to explain what happens in the case that the domain metric of the flow degenerates at infinite time.
We obtain a branched minimal immersion from the degenerate domain.
\end{abstract}

\section{Introduction}
Let $M:=M_\gamma$ be a smooth closed orientable surface of genus $\gamma\geq 2$, let $N=(N,G)$ be a smooth compact Riemannian manifold of any dimension, and let $\eta>0$ be some
fixed parameter.
We consider the flow
\begin{equation}
\label{flow}
\pl{u}{t}=\tau_g(u);\qquad \pl{g}{t}=\frac{\eta^2}{4} Re(P_g(\Phi(u,g))),
\end{equation}
introduced in \cite{R-T} as the natural gradient flow of the harmonic map energy when both a map $u:M\to (N,G)$ and a hyperbolic metric $g$ on its domain $M$ are allowed to evolve.
Here, $\tau_g(u)$ represents the tension field of $u$ (i.e.
$\tr \nabla du$), $P_g$ represents the $L^2$-orthogonal projection from the space of quadratic differentials on $(M,g)$ onto the space of \emph{holomorphic} quadratic differentials, and
$\Phi(u,g)$ represents the Hopf differential -- see \cite{R-T} for further information.
The flow decreases the energy according to
\beq
\label{energy-identity}
\frac{dE}{dt}=-\int_M |\tau_g(u)|^2+\left(\frac{\eta}{4}\right)^2 |Re(P_g(\Phi(u,g)))|^2.
\eeq
While a solution to \eqref{flow} can be projected down to give a path in 
Teichm\"uller space, it is worth digesting that the flow is not to be considered to be the flow of a map coupled with a flow in 
(finite dimensional) Teichm\"uller space. See \cite{R-T} for more details of the geometry behind the equations.

Given any initial data $(u_0,g_0)\in H^1(M,N)\times \M$, 
with $\M$ the set of smooth hyperbolic metrics on $M$, we know \cite{rupflin_existence} that a (weak) solution of \eqref{flow} exists on a maximal interval $[0,T)$, 
smooth except possibly at finitely many times,
and that $T<\infty$ only if the flow of metrics degenerates in moduli space
as $t\nearrow T$, that is if the length $\ell(g(t))$ of the shortest closed geodesic $\ell(g(t))\to 0$ as $t\nearrow T$.

In \cite{R-T} we proved that if such a degeneration does not occur, not even as $t\to \infty$, then the maps $u(t)$ subconverge (after reparametrisation) to a branched minimal immersion (or a constant map)
with the same action on $\pi_1$ as the initial map $u_0$.

Here we prove that also in the case that the metric degenerates as $t\to \infty$ (but not before) we also obtain asymptotic convergence to a minimal object in the following sense

\begin{thm}\label{thm:asymptotics}
Suppose that $(u,g)$ is a global (weak) solution of \eqref{flow}
as described above for which $\ell(g(t))\to 0$ as $t\to \infty$. Then there exist a sequence of times $t_i\to \infty$, a number $1\leq k\leq 3(\gamma-1)$ and a
hyperbolic punctured surface $(\Si,h,c)$ with $2k$ punctures (and possibly disconnected) such that the following holds.
\begin{enumerate}
\item The surfaces $(M,g(t_i),c(t_i))$ converge to the surface $(\Si,h,c)$ by collapsing $k$ simple closed geodesics $\sigma^{j}_{i}$
in the sense of Proposition \ref{pro1.1}; in particular there is a sequence of diffeomorphisms $ f_i:\Si\to M\setminus \cup_{j=1}^k \sigma^{j}_{i}$
such that $$f_i^*g(t_i)\to h \text{ and } f_i^*c(t_i)\to c \text{ smoothly locally, }$$
where $c(t)$ denotes the complex structure of $(M,g(t))$.
 \item The maps $f_i^*u(t_i):=u(t_i)\circ f_i$ converge to a limit $u_\infty$ weakly in $H^1_{loc}(\Si)$ and weakly in $H_{loc}^2(\Si\setminus S)$ as well as strongly in $W_{loc}^{1,p}(\Si\setminus S)$,
$p\in [1,\infty)$,
away from a finite set of points $S\subset \Si$ at which
energy concentrates.
\item 
The limit $u_\infty:\Si\to N$ extends to a branched minimal immersion (or constant map) on each component of the compactification of $(\Sigma,c)$ obtained by filling in each of the $2k$ punctures.
\end{enumerate}
\end{thm}

The issue of finite time degeneration of the metric component of solutions to \eqref{flow} as well as the existence of \textit{global} (generalised) solutions
of \eqref{flow} for arbitrary initial data will be discussed in future work.

Key for the proof of this result, which is given in Section \ref{mainthmpf}, is a good understanding of the structure of the space of holomorphic quadratic differentials on a sequence of degenerating hyperbolic surfaces.
This subject has been investigated from many points of view; here we develop the geometric analysis side of the theory,
identifying precisely a subspace $W_i$ of the space of holomorphic quadratic differentials that persists in this limit, and
obtaining quantitative estimates on those differentials which centre around Lemma \ref{Wi_structure}. Conversely, the holomorphic quadratic differentials orthogonal to $W_i$
concentrate on degenerating collars as $i\to\infty$. This fact will be made more precise in the upcoming paper
\cite{RT2} where we establish a uniform Poincar\'e-type estimate for general quadratic differentials on hyperbolic surfaces of bounded genus.

A starting point to understand the basic theory of degenerating hyperbolic surfaces, and to digest our notation, is the appendix.

\emph{Acknowledgements:} This work was partially supported by The Leverhulme Trust.

\section{The space of holomorphic quadratic differentials on degenerating surfaces}
In this section we would like to consider the space of holomorphic quadratic differentials on a closed surface $M$ of genus $\gamma\geq 2$, with respect to a degenerating sequence of complex structures
and corresponding hyperbolic metrics $g_i$.
For each complex structure, we will be viewing the space $\Hol(M,g_i)$ of holomorphic quadratic differentials as a complex vector space
of dimension $3\ga-3$, equipped with various $L^p$ norms and the $L^2$ inner product
arising from the standard Hermitian inner product on each fibre computed with respect to 
$g_i$. 

According to the Deligne-Mumford compactness theorem in the form outlined in Proposition \ref{pro1.1},
after passing to a subsequence and pulling back by diffeomorphisms
the surfaces $(M,g_i$) converge to a limit $(\Si,h)$ that is a  hyperbolic punctured surface by pinching
$1\leq k\leq 3\gamma-3$ collars.

On the limit, we will need to consider the space $\Hol(\Si,h)$ of all holomorphic quadratic differentials that lie in $L^1(\Si,h)$.  If we extend such a holomorphic quadratic differential across all the punctures on the limit then the singularity can at worst be a simple pole
(by virtue of it lying in $L^1$), see Lemma \ref{lemma:poles-norms}. Thus by Riemann-Roch, the (complex) dimension of $\Hol(\Si,h)$ is $3(\gamma-1)-k=\dim_\C(\Hol(M,g_i))-k$.

One central task in this paper is to isolate a sequence of subspaces $W_i\subset \Hol(M,g_i)$ of complex dimension $3(\ga-1)-k$ that converge in some sense to $\Hol(\Si,h)$ without loss of any $L^p$ norm. These subspaces can be loosely characterised as consisting of the holomorphic quadratic differentials that decay rapidly along each degenerating collar -- see in particular Lemma \ref{Wi_structure}.
Orthogonal to that subspace is a complementary subspace of holomorphic quadratic differentials that concentrate entirely on degenerating collars (in terms of $L^2$ norm) and thus have vanishing $L^1$ norm in the limit. The analysis of these latter subspaces will be refined in forthcoming work \cite{RT2}.

First of all we prove that $L^2$-bounded holomorphic quadratic differentials have a form of $L^1$ compactness.

\begin{lemma}
\label{L1atleastepsilon}
Suppose $M$ is a closed surface of genus $\gamma\geq 2$.
Suppose $g_i$ is a sequence of degenerating hyperbolic metrics on $M$ as described in Proposition \ref{pro1.1} and $\Th_i$ is a sequence of holomorphic quadratic differentials (with respect to $g_i$) satisfying $\|\Th_i\|_{L^2(M,g_i)}=1$.
Then after passing to a subsequence, we have
$$f_i^*\Th_i\to\Th_\infty,$$
smoothly locally on $\Si$, where $\Th_\infty$ is a
holomorphic quadratic differential on $(\Si,h)$, lying in
$\Hol(\Si,h)$. Moreover, we have
\begin{equation}
\label{L1conv}
\|\Th_i\|_{L^1(M,g_i)}\to\|\Th_\infty\|_{L^1(\Si,h)}
\end{equation}
as $i\to\infty$, and in particular, if in addition we know that
$\|\Th_i\|_{L^1}\geq \epsilon>0$,
then $\Th_\infty$ is not identically zero.

More generally, if $L_i$ is a sequence of $n$-dimensional subspaces of the complex vector spaces 
$\Hol(M,g_i)$ of holomorphic quadratic differentials (for some $n$) with the property that
\begin{equation}
\label{noL1decay}
\inf_{v\in L_i}\frac{\|v\|_{L^1(M,g_i)}}{\|v\|_{L^2(M,g_i)}}\geq \ep>0,
\end{equation}
then there exists a $n$-dimensional subspace $L_\infty$ of $\Hol(\Si,h)$ such that after passing to a subsequence, the subspaces $L_i$ converge to $L_\infty$ in the sense that there exists a basis
$\Th_\infty^j$ of $L_\infty$, and for each $i$ there exists an $L^2$-unitary basis
$\Th_i^j$ of $L_i$ such that for each $j\in \{1,\ldots,n\}$, we have
$f_i^*\Th_i^j \to \Th_\infty^j$ smoothly locally on $\Sigma$ as
$i\to\infty$.
\end{lemma}

Note that although the $L^1$ norm passes to the limit above,
we could lose $L^2$ norm along a collar; the basis $\Th_\infty^i$ obtained in the limit is thus in general not unitary.
Later, once we have constructed the subspaces $W_i$ mentioned briefly earlier, we will see that sequences within $W_i$ will also enjoy convergence of
their $L^2$ norms and indeed of all $L^p$ norms, $p\in[1, \infty]$, and thus that unitary families of elements in $W_i$ subconverge again to a unitary family in $\Hol(\Si,h)$.

We remark that as a consequence of the above lemma we could obtain the existence of subspaces of $\Hol(M,g_i)$ of dimension $k$ on which the ratio of the  $L^1$ and $L^2$ norms is decaying
to zero as $i\to \infty$. A more refined version of this statement is shown and needed in \cite{RT2}.

\begin{proof} (Lemma \ref{L1atleastepsilon}.)
For the first part of the lemma, the smooth local convergence follows simply from $L^2$-boundedness and holomorphicity of $\Th_i$
as we now explain. Given any compact subset $K$ of $\Sigma$ we can
choose $\de>0$ small enough such that for every $i\in \N$ we have $K\subset \Si_i^\de$, the $\de$-thick part of the surfaces $(\Si,f_i^*g_i)$.
On the $\de$-thick parts of the surfaces the $C^m$ norms of \emph{holomorphic} quadratic differentials
are controlled uniformly by their $L^1$ norm
$$\norm{f_i^*\Theta_i}_{C^m(K)}\leq \norm{f_i^*\Th_i}_{C^m(\Si_i^\de)}
\leq 
C_\de \norm{\Th_i}_{L^1(M,g_i)},$$
by Lemma \ref{lemma:holomorphic} in the appendix.
Using Arzela-Ascoli and taking a diagonal sequence we obtain smooth local convergence of a subsequence of $f_i^*\Th_i$ to a limit $\Theta_\infty$.
Since also the complex structures of $(\Si, f_i^*g_i)$
converge smoothly locally
(see Appendix \ref{DMsubsection}) to the complex structure $c$ of the limit surface,
the limit $\Th_\infty$ is a holomorphic quadratic differential on $(\Si,h,c)$.

Based on the uniform bound on the $L^2$ norms we can now show convergence of $L^1$ norms.
It is a consequence of the Collar Lemma \ref{Collar-lemma} --
see Lemma \ref{lem2.3}
-- that the area of the
$\de$-thin part of $(M,g_i)$ is less than $C\de$ for a uniform constant $C$ independent of $i$ and $\de>0$.
By Cauchy-Schwarz, we may thus estimate over this $\de$-thin part
$$\|\Th_i\|_{L^1(M\backslash M^\de_i,g_i)}\leq C\sqrt{\de}\|\Th_i\|_{L^2(M,g_i)} \leq C\sqrt{\de}.$$
Meanwhile, 
by Lemma \ref{lemma:thick}
we have convergence of the norms on the $\de$-thick parts $M_i^\de$ of $(M,g_i)$
$$\|\Th_i\|_{L^1(M_i^\de,g_i)}\to \|\Th_\infty\|_{L^1(\Si^\de,h)}.$$
Thus
$$\limsup_{i\to\infty} \|\Th_i\|_{L^1(M)}
=\limsup_{i\to\infty} (\|\Th_i\|_{L^1(M^\de_i)}
+ \|\Th_i\|_{L^1(M\backslash M^\de_i)})
\leq \|\Th_\infty\|_{L^1(\Si^\de)} + C\sqrt{\de}
\to \|\Th_\infty\|_{L^1(\Si)} $$
as $\de\downto 0$, and
$$\liminf_{i\to\infty} \|\Th_i\|_{L^1(M)}
\geq \liminf_{i\to\infty}\|\Th_i\|_{L^1(M^\de_i)}
= \|\Th_\infty\|_{L^1(\Si^\de)}
\to \|\Th_\infty\|_{L^1(\Si)} $$
as $\de\downto 0$, and so we have proved
\eqref{L1conv}.
Note that a by-product of this is that $\Th_\infty$ must lie in
$L^1$, and therefore in $\Hol(\Si,h)$. We have proved the first part of the lemma.

It remains to prove the second part of the lemma.
To do this, for each $i$, pick any unitary basis
$\Th_i^j$ of $L_i$ ($j=1,\ldots,n$). By assumption \eqref{noL1decay}, we have $\|\Th_i^j\|_{L^1}\geq \ep$ for all $i$. By the first part of the lemma, after passing to a subsequence in $i$, we may assume that there exist nonzero limits
$\Th^j_\infty\in \Hol(\Si,h)$ of the sequences $f_i^*(\Th_i^j)$.
It remains to show that these limits span an $n$-dimensional
subspace of $\Hol(\Si,h)$. If that were not the case, then we could find a unitary vector $b\in\C^n$ such that
$$\sum_{j=1}^n b_j \Th^j_\infty=0.$$
But then we could consider the sequence
$$\tilde\Th_i:=\sum_{j=1}^n b_j \Th^j_i$$
of unitary vectors in $L_i$, which converges smoothly locally
to zero by construction.
But by the first part of the lemma, and by the assumed lower bound on the $L^1$ norm from \eqref{noL1decay}, this must converge to a \emph{nonzero} limit.
\end{proof}

\subsection{Analysis of holomorphic quadratic differentials on collar regions}

The subspaces $W_i\subset \Hol(M,g_i)$ alluded to earlier will be defined in terms of the behaviour of elements along so-called collar regions, and we now discuss the geometry of collars.

Following Lemma \ref{Collar-lemma},
let $\cc(\ell)$
be the hyperbolic collar around a simple closed geodesic of length $\ell$, i.e. a region $(-X(\ell),X(\ell))\times S^1$ parametrised by local conformal coordinates
$(s,\theta)$ (or a complex coordinate $w=s+i\th$)
and equipped with the metric
$$\rho^2(ds^2+d\th^2), \qquad\text{ where }\rho=\frac{\ell}{2\pi \cos(\frac{\ell s}{2\pi})}.$$
We can equivalently think of the collar as an annulus $D_{e^X}\backslash D_{e^{-X}}$ in the complex plane parametrised by
$z:=e^w$ and equipped with an appropriate metric $\tilde\rho^2 dzd\bar z$.
A holomorphic quadratic differential $\Phi$ on the collar is given  in these coordinates by $\Phi=\phi(w)dw^2=\tilde\phi(z)dz^2$ for holomorphic functions $\phi(w)=z^2\tilde\phi(z)$ on the
cylinder respectively the annulus.

Decomposing $\tilde \phi$ as a Laurent series
$$\tilde\phi(z)=\sum_{n=-\infty}^\infty \tilde b_{n} z^n,$$
converging uniformly away from the boundary of the annulus, the function $\phi$ representing the holomorphic quadratic differential in the cylindrical coordinates can thus be
written as
\beq \label{eq:Laurent}
\phi(s,\th)=\sum_{n=-\infty}^\infty b_n e^{nw}
=\sum_{n=-\infty}^\infty b_n e^{ns}\,e^{in\th},\eeq
with $b_n:=\tilde b_{n-2}\in \C$.

We will split this function $\phi$, and hence the quadratic differential $\Phi$ on the collar, into its \emph{principal part}
$\phi_0 dw^2:=b_0(\Phi) dw^2=\tilde b_{-2}z^{-2}dz^2$ and the remaining, \emph{collar-decay} part
$\Phi-b_0(\Phi)dw^2$.

Because each of the terms in the sum \eqref{eq:Laurent} are $L^2$-orthogonal, even when restricted to circles $\{s\}\times S^1$,
the components $b_0(\Phi)dw^2$ and $\Phi-b_0(\Phi)dw^2$ are orthogonal on the collar with respect to the hyperbolic metric, and even on any sub-collar $(s_1,s_2)\times S^1\subset \cc(\ell)$.

We control the collar-decay components with the following lemma, which will be proved at the end of the section.

\begin{lemma}
 \label{lemma:collar}
Holomorphic quadratic differentials with zero principal part decay rapidly along the collars in the following uniform sense. There exist numbers $\de_0>0$ and $C<\infty$
such that any holomorphic quadratic differential $\Th$ on a collar $\cc(\ell)$, $0<\ell<2\arsinh(1)$, with principal part $b_0(\Th)dw^2=0$ satisfies
$$\|\Th \|_{L^\infty(\de-\text{thin} (\cc(\ell)))}\leq C\cdot e^{-\pi/\delta}\delta^{-2}\|\Th \|_{L^2(\de_0-\text{thick}(\cc(\ell)))}$$
for all numbers $0<\delta\leq \delta_0$.
\end{lemma}
\begin{rmk}
Given a holomorphic quadratic differential on a hyperbolic surface $(M,g)$ satisfying the assumptions of the above lemma on such a collar neighbourhood $\Col\subset (M,g)$, it is useful to observe that
Lemma \ref{lemma:collar} implies an estimate of the form
$$\|\Th \|_{L^\infty(\de-\text{thin} (\cc))}\leq C\cdot e^{-\pi/\delta}\delta^{-2}\|\Th \|_{L^1(M,g)},$$
since the $L^2$ norm over the thick part of the surface is controlled in terms of the $L^1(M,g)$ norm by Lemma \ref{lemma:Lp}.
\end{rmk}

The proof of this lemma will be given at the end of the section.
The quadratic differentials lying in $W_i$, which we will now define,  will be of the type in Lemma \ref{lemma:collar} on each degenerating collar.

\begin{lemma}(Introducing $W_i$.)\label{Wi_structure}
Given a sequence of hyperbolic surfaces $(M,g_i)$ degenerating to $(\Si,h)$ by
collapsing $k$ collars $\Col_i^j\cong\Col(\ell_i^j)$
as described in Proposition \ref{pro1.1},
we let
\beq \label{def:Wi} W_i:=\{\Th\in\Hol(M,g_i):\, b_0^{j}(\Th)dw^2=0 \text{ for every } j\in\{1\ldots k\}\,\}\eeq
be the subspace of holomorphic quadratic differentials that have vanishing principal part on every degenerating collar $\cc_i^j$, $j\in \{1\ldots k\}$.
Then:
\begin{enumerate}
\item[(i)] The elements of $W_i$ decay rapidly along the collar regions in the sense that for every $\delta>0$ and every $i\in\N$
\beq
\sup_{w\in W_i}\frac{\|w\|_{L^\infty(M\backslash M^\de_i,g_i)}}
{\|w\|_{L^2(M,g_i)}}\leq C\cdot \delta^{-2}e^{-\pi/\delta}
\eeq
for a uniform constant $C<\infty$ independent of $i$ and $\de$, where $M_i^\de=\de\thick(M,g_i)$.
\item[(ii)]
There exists $I_0\in \N$ such that for $i\geq I_0$,
$W_i$ is a $3(\gamma-1)-k$ dimensional subspace
of $\Hol(M,g_i)$.
\item[(iii)]
$W_i$ converges to $\Hol(\Si,h)$ in the sense that
for every $i\geq I_0$ there exists an $L^2$-unitary
basis $\{\Theta_i^j\}_{j=1}^{3\gamma-3-k}$ of $W_i$ that converges
$$f_i^*\Theta_i^j\to \Theta_\infty^j\in \Hol(\Si,h) \text{ as } i\to \infty$$
smoothly locally to an $L^2$-unitary basis of $\Hol(\Si,h)$.
Furthermore, the convergence preserves all $L^p$ norms, $p\in [1,\infty]$ in the sense that if a general sequence
of elements $\Th_i\in W_i$ converges
$f_i^*\Th_i\to \Th_\infty$ locally, then the limit $\Th_\infty$ is in $\Hol(\Si,h)$ and
for each $p\in[1,\infty]$ we have $\norm{\Th_\infty}_{L^p(\Si,h)}=\lim_{i\to\infty}\norm{\Th_i}_{L^p(M,g_i)}<\infty$.
\end{enumerate}
\end{lemma}

\begin{proof} (Lemma \ref{Wi_structure}.)
Part (i) of the lemma
is an immediate consequence of the definition of $W_i$, the
key Lemma \ref{lemma:collar} about the behaviour of holomorphic quadratic differentials on collar regions and the fact that for $\de>0$ sufficiently small, the $\de$-thin part of the surface $(M,g_i)$ is
contained in the union of the collar regions $\cc_i^j$ -- see Proposition \ref{Thick-thin-2} in the appendix.

Using Part (i), we deduce that for $\de>0$ chosen sufficiently small,
$$\norm{w}_{L^2(M_i^\de,g_i)}\geq (1-C\de^{-2}e^{-\pi/\de})\norm{w}_{L^2(M,g_i)}\geq \frac12\norm{w}_{L^2(M,g_i)} \text{ for all } w\in W_i, i\in\N$$
which together with Lemma \ref{lemma:Lp}
implies a uniform lower bound on the $L^1$ norm of elements of $W_i$ of
\begin{equation}
\label{l_bd}
\inf_{w\in W_i}\frac{\norm{w}_{L^1(M,g_i)}}{\norm{w}_{L^2(M,g_i)}}\geq\frac12\inf_{w\in W_i}\frac{\norm{w}_{L^1(M,g_i)}}{\norm{w}_{L^2(M_i^\de,g_i)}}\geq \eps>0
\end{equation}
valid for all $i\in \N$.
We claim that there exists $I_0\in\N$ such that
$\dim(W_i)=3(\gamma-1)-k$ for all $i\geq I_0$.
Note that by definition,
$$\dim(W_i)\geq \dim(\Hol(M,g_i))-k=\dim(\Hol(\Si,h))=3(\gamma-1)-k,$$
so the only alternative to our claim is if, after passing to a subsequence, we have $\dim(W_i)=m>\dim(\Hol(\Si,h))$ for
each $i$.
By Lemma \ref{L1atleastepsilon}, using \eqref{l_bd}, we conclude that the spaces $W_i$ subconverge to a subspace of $\Hol(\Si,h)$ of the same dimension $m>\dim(\Hol(\Si,h))$ in the sense described in that lemma, which is impossible. This proves the claim, i.e. Part (ii) of the lemma.

Now that we know the dimension of $W_i$, for $i\geq I_0$, we can apply Lemma \ref{L1atleastepsilon} again to obtain, after taking a subsequence,  a sequence of unitary bases $\Th_i^j$ and a limit basis $\Th_\infty^j$. To prove Part (iii) of the lemma, even allowing ourselves to take this subsequence, we still have to show that this limit is unitary, and more generally that
all $L^p$ norms ($p\in[1,\infty]$) are preserved during local convergence as described in the lemma, which will follow from the rapid decay of the elements of $W_i$ on collars.

Let $\Th_i$ be any sequence of elements of $W_i$ that converges smoothly locally $f_i^*\Th_i\to \Th_\infty$. Then $\Th_\infty$ is again a holomorphic quadratic differential which,
as we shall prove now, has finite $L^1$ norm and is thus an element of $\Hol(\Si,h)$.

We recall
that for any $\de>0$ the $\de$-thick part $\Si_i^\de$ of $(\Si,f_i^*g_i)$ converges as described in Lemma \ref{lemma:thick} to the compact set $\Si^\de$ and thus that
for every $p\in [1,\infty]$
$$\norm{\Th_i}_{L^p(M_i^\de,g_i)}\to\norm{\Th_\infty}_{L^p(\Si^\de,h)},$$
by the smooth local convergence.

Let now $\de_0>0$ be a fixed number as in Lemma \ref{lemma:collar}. Since $\norm{\Th_\infty}_{L^2(\Si^{\de_0},h)}$ is bounded, we have a uniform bound on
$\norm{\Th_i}_{L^2(M_i^{\de_0},g_i)}$ so that according to Lemma \ref{lemma:collar} the norms over the thin part of the surfaces $(M,g_i)$
$$\norm{\Th_i}_{L^p(M\setminus M_i^\de,g_i)}\leq C\cdot \norm{\Th_i}_{L^\infty(M\setminus M_i^\de,g_i)} \leq C\de^{-2}e^{-\pi/\de}\to 0  $$
converge to zero uniformly in $i$  and $p\in[1,\infty]$ as $\de\to 0$.

Thus the global $L^p$ norms are bounded uniformly and converge
 $$\norm{\Th_\infty}_{L^p(\Si,h)}=\sup_{\de>0}\norm{\Th_\infty}_{L^p(\Si^\de,h)}=
\sup_{\de>0}\lim_{i\to\infty} \norm{\Th_i}_{L^p(M_i^\de,g_i)}=\lim_{i\to\infty} \norm{\Th_i}_{L^p(M,g_i)}$$
for every $p\in [1,\infty]$.

The preservation particularly of the $L^2$ norm in the above convergence implies that the basis $\{\Th_\infty^j\}$ of $\Hol(\Si,h)$ that is obtained above as a limit of (a subsequence of) unitary bases $\{\Th_i^j\}$ of $W_i$, is again unitary.

That would complete the proof of Part (iii) of the lemma, except that we have allowed ourselves to take a subsequence above.
However, it is easy to return to the original (pre-subsequence) sequence and take an extended sequence of unitary bases $\{\Th_i^j\}$, and check that after modifying them by a sequence of unitary transformations, they converge to $\{\Th_\infty^j\}$: If not, then let us take a subsequence which, however we modify with a sequence of unitary transformations, stays outside some neighbourhood of $\{\Th_\infty^j\}$. Following the argument above,
we may pass to a further subsequence to get convergence to some other unitary limit basis $\{\tilde \Th_\infty^j\}$, but then by modifying this whole subsequence by an appropriate fixed unitary transformation, we get convergence to $\{\Th_\infty^j\}$, which is a contradiction.

\end{proof}

\begin{rmk}
In our considerations above, we have derived properties of elements of $W_i$ by using the fact that their principal parts on each collar vanish. In practice, this can be weakened. For example, to have preservation of the $L^2$ norm in Part (iii) above, we would only need that
\beq \label{ass:bij} b_0^{j}(\Th_i)\cdot (\ell_i^{j})^{-\frac{3}{2}}\to 0 \text{ as }i\to\infty \text{ for each } j\in\{1\ldots k\}.\eeq
This is because we only need that the principal parts of $\Th_i$ are vanishing in the sense that the $L^2$ norms of 
$b_0^{j}(\Th_i)dw^2$ converge to zero, and if we adopt the normalisation convention that $\abs{dw^2}=2\rho^{-2}$, we have, for
$0<\de<\arsinh(1)$ and $\ell\in (0,2\de)$,
\beqa \label{est:delta-thin2}
\|dw^2\|_{L^2(\de\text{-thin}(\cc(\ell)))}^2&=2\pi\int_{-X_\delta}^{X_\delta} |dw^2|^2\rho^2 ds=8\pi\int_{-X_\delta}^{X_\delta} \rho^{-2} ds
&=\frac{C}{\ell^2} \int_{-X_\delta}^{X_\delta}\cos^2\left(\frac{\ell s}{2\pi}\right) ds\\
&=C_0\ell^{-3}+O(\de^{-3})\eeqa
for a constant $C_0>0$ independent of $\ell$ and $\de$,
and where
\beq \label{est:Xdelta}
X_\de(\ell)=\frac{\pi^2}{\ell}-\frac{2\pi}{\ell}\arcsin\left(\frac{\sinh(\ell/2)}{\sinh(\de)}\right)=\frac{\pi^2}{\ell}-\frac{\pi}{\de}+O(1)\eeq
was defined in Lemma \ref{lem2.3} so
that the $\de$-thin part of a collar $\cc(\ell)$ is given by $(-X_\de(\ell),X_\de(\ell))\times S^1$.
\end{rmk}

We finally give the proof of the key Lemma \ref{lemma:collar} on the decay on collars of holomorphic quadratic differentials with zero principal part.
\begin{proof} (Lemma \ref{lemma:collar}.)
Let $0<\de_0<\arsinh(1)$ be a constant to be fixed later on. Given any number $0<\ell<2\arsinh(1)$, we consider the corresponding collar region $\cc(\ell)$ around a geodesic of length $\ell$ as described in the Collar Lemma \ref{Collar-lemma}.
We first remark that since
the $\de_0$-thin part of $\cc(\ell)$ is empty in the case that $\ell\geq 2\de_0=:\ell_0$,  we may restrict our attention to values $0<\ell<\ell_0$. We recall that the subcylinder
$(-X_{\de_0}(\ell),X_{\de_0}(\ell))\times S^1$ describing the $\de_0$-thin part of the collar $\cc(\ell)=(-X(\ell),X(\ell))\times S^1$ is characterised by \eqref{est:Xdelta} and is
thus bounded away uniformly from the boundary of $\cc(\ell)$, say
$$X(\ell)-X_{\de_0}(\ell)\geq 1 \text{ for all } 0<\ell<\ell_0,$$
if $\de_0$ is initially chosen small enough. This is the only constraint we impose on $\de_0$.

Let now $\Theta$ be a holomorphic quadratic differential on such a collar $\cc(\ell)$, $0<\ell\leq \ell_0$, without loss of generality normalised to satisfy $\norm{\Th}_{L^2(\de_0-\text{thick}(\cc(\ell)))}=1$,
and suppose that $\Theta$ has zero principal part, i.e. that it is given by a converging sum of the form
$$\Th=\sum_{n\in\Zn} b_n e^{ns}\,e^{in\th}dw^2.
$$
Using that all terms in this sum are $L^2$-orthogonal on subcylinders 
as well as that $\abs{dw^2}=2\rho^{-2}$ (with our normalisation)
we thus obtain

\beqa 1&=\norm{\Theta}_{L^2(\de_0-\text{thick}(\cc(\ell)))}^2
=\sum_{n\in \Zn}\abs{b_n}^2\norm{e^{ns}dw^2}_{L^2(\de_0-\text{thick}(\cc(\ell)))}^2\\
&\geq 8\pi \sum_{n\in \N} \left[\abs{b_n}^2\int_{X(\ell)-1}^{X(\ell)}e^{2ns}\rho^{-2}(s) ds+\abs{b_{-n}}^2\int_{-X(\ell)}^{-X(\ell)+1}e^{-2ns}\rho^{-2}(s) ds\right]\\
&\geq c \sum_{n\in\Zn}\abs{b_n}^2 \abs{n}^{-1}e^{2\abs{n}X(\ell)}\label{est:bn}
\eeqa
for a uniform constant $c>0$.
Here the last inequality follows from the uniform upper bound $\rho(s)\leq C$ for the conformal factor on the ends of collars around geodesics of bounded length.

We conclude in particular that $\abs{b_n}\leq C\sqrt{\abs{n}}e^{-\abs{n}X(\ell)}$ for every $n\in\Zn$, resulting in a bound of
\beq \label{est:Linfty}
\norm{\Th}_{L^\infty(\de\text{-thin}(\cc(\ell)))}\leq C\sum_{n\in \Zn} \sqrt{n}e^{-\abs{n}X(\ell)}\sup_{s\in [-X_\delta(\ell),X_\delta(\ell)]}
e^{ns}\abs{dw^2}(s),
\eeq
valid for arbitrary values of $0<\de\leq \de_0$ and with a uniform constant $C<\infty$.
We now remark that for each $\ell$ and each $n\in \Zn$ the function
$$s\mapsto  e^{ns}\abs{dw^2}(s)=2e^{ns}\rho^{-2}(s)=\frac{8\pi^2}{\ell^2}\cdot \cos^2\left(\frac{\ell s}{2\pi}\right)e^{ns}$$
is monotone on the whole interval $(-X(\ell),X(\ell))$ so the supremum in \eqref{est:Linfty} is achieved at one end of the $\de$-thin part of the collar. We obtain the desired bound of
$$
\norm{\Th}_{L^\infty(\de\text{-thin}(\cc(\ell)))}\leq C\sum_{n\in \Zn} \sqrt{n}e^{-\abs{n}(X(\ell)-X_\de(\ell))}\rho^{-2}(X_\de(\ell))\leq Ce^{-\frac{\pi}{\de}}\de^{-2}.
$$

\end{proof}

\subsection{Applications of the structure theory for holomorphic quadratic differentials}

As a consequence of the results derived in the previous section we will now obtain a continuity result for the projection of general quadratic differentials onto the subspaces $W_i\subset \Hol(M,g_i)$
of holomorphic quadratic differentials described in Lemma \ref{Wi_structure}.
In a sense to be made precise, the projections onto $W_i$
will converge to the projection onto the entire space
$\Hol(\Si,h)$ of integrable holomorphic quadratic differentials on the limit.

Before stating this result let us first recall that the space of holomorphic quadratic differentials on a compact surface $(M,g)$ is finite dimensional and that all its elements are bounded.
Thus the $L^2(M,g)$-orthogonal projection $P_g=P_g^{\Hol(M,g)}$ onto $\Hol(M,g)$ satisfies
an estimate of the form $\norm{P_g(\Psi)}_{L^2(M,g)}\leq C\norm{\Psi}_{L^1(M,g)}$, where $\Psi$ is any quadratic differential
with finite $L^2$ norm,  
and can thus be extended continuously to a projection from the space of all quadratic differentials with finite $L^1$ norm to $\Hol(M,g)$,
which we still denote by $P_g$.

Similarly, by virtue of the $L^\infty$ bounds on integrable holomorphic quadratic differentials given in Lemma \ref{lemma:poles-norms}, we can extend the $L^2(\Si,h)$-orthogonal projection $P_h^{\Hol(\Si,h)}$ to
the space of all quadratic differentials on the limit surface that have finite $L^1$ norm.

We furthermore remark that
given a sequence $\Psi_i$ of quadratic differentials on degenerating surfaces we can think of the sequence $f_i^*\Psi_i$
either as a sequence of
quadratic differentials with respect to the varying metrics $f_i^*g_i$ or as a general sequence of (complex)
tensors on the fixed Riemannian surface $(\Si,h)$,
thus allowing
us to talk about convergence of these tensors say in $L^1_{loc}(\Si,h)$.
Furthermore, any tensor $\Psi_\infty$ obtained as a limit of a sequence of quadratic differentials
$f_i^*\Psi_i\to \Psi_\infty$ in $L_{loc}^1(\Si,h)$ is again a quadratic differential now with respect to $(\Si,h)$ owing to the smooth local convergence of the complex structures
(as in Section \ref{DMsubsection}).
If the norms $\norm{\Psi_i}_{L^1(M,g_i)}$ are bounded and thus also $\norm{\Psi_\infty}_{L^1(\Si,h)}<\infty$, the projection of the limit $\Psi_\infty$ to the space $\Hol(\Si,h)$ is well defined, as remarked above,
and we can discuss the continuity of the projections in the following sense.

\begin{thm}
\label{neck_lemma}
Let $(M,g_i)$ be a sequence of degenerating hyperbolic surfaces converging to a  hyperbolic punctured surface $(\Si,h)$ by collapsing $k$ collars as described in Proposition \ref{pro1.1}
and
let $W_i\subset\Hol(M,g_i)$ be the $3(\gamma-1)-k$ dimensional subspace defined in Lemma \ref{Wi_structure} that consists of elements of $\Hol(M,g_i)$ decaying rapidly on the degenerating
collars.
\begin{enumerate}
 \item[(i)] Suppose we have a sequence of quadratic differentials $\Psi_i$ on $(M,g_i)$ satisfying a uniform bound $\|\Psi_i\|_{L^1(M,g_i)}\leq C$ 
which converges
$$f_i^*\Psi_i\to\Psi_\infty \text{ in } L^1_{loc}(\Si,h).$$
Suppose further that we have a sequence of holomorphic quadratic differentials $\Th_i\in W_i$
such that
$$f_i^*\Th_i\to\Th_\infty$$
smoothly locally.
Then
\beq\int_M \langle \Psi_i,\Th_i\rangle d\mu_{g_i}\to
\int_\Sigma \langle \Psi_\infty,\Th_\infty\rangle d\mu_{h}.\label{eq:neck-lemma}\eeq
\item[(ii)] 
The
$L^2(M,g_i)$-orthogonal projection $P_{g_i}^{W_i}$ onto $W_i$ converges to the 
$L^2(\Si,h)$-orthogonal projection $P_h^{\Hol(\Si,h)}$ onto the space of integrable holomorphic
quadratic differentials $\Hol(\Si,h)$ on the limit surface
in the following sense:

For any sequence $\Psi_i$ of quadratic differentials on $(M,g_i)$ with uniformly bounded $L^1$ norms that converges $f_i^*\Psi_i\to \Psi_\infty$
locally in $L^1(\Si,h)$ the projections converge
$$f_i^*(P_{g_i}^{W_i}(\Psi_i))\to P_h^{\Hol(\Si,h)}(\Psi_{\infty}) \text{ smoothly locally }$$ while preserving any $L^p$ norm
$$\lim_{i\to \infty}\norm{P_{g_i}^{W_i}(\Psi_i)}_{L^p(M,g_i)}= \norm{P_h^{\Hol(\Si,h)}(\Psi_\infty)}_{L^p(\Si,h)}
, \qquad 1\leq p\leq \infty.$$
\end{enumerate}
\end{thm}

The first part of the lemma will be used in the proof of the second part, which in turn is required in the proof of the main Theorem \ref{thm:asymptotics}.
\begin{proof} [Proof of Theorem \ref{neck_lemma}]
Let $\Psi_i$ and $\Th_i$ be as in Theorem \ref{neck_lemma} (i)
and recall that by the remarks made earlier in this section all objects in the theorem are well defined.
Given any $\delta>0$, we now use that the $\delta$-thick part $\Si_i^\de$ of $(\Sigma,f_i^*g_i)$ converges as described in Lemma \ref{lemma:thick} to the (compact) $\de$-thick part $\Si^\de$ of the limit surface.
Combined with the smooth local convergence of the metrics and the local $L^1(\Si,h)$ convergence of
$\langle\Psi_i,\Th_i\rangle\circ f_i=\langle f_i^*\Psi_i,f_i^*\Th_i\rangle\to \langle\Psi_\infty,\Th_\infty\rangle$ we thus find that
\beq \label{eq:convergence}
\int_{M_i^\delta}\langle\Psi_i,\Th_i\rangle d\mu_{g_i}=\int_{\Sigma_i^\delta}\langle f_i^*\Psi_i,f_i^*\Th_i\rangle d\mu_{f_i^*g_i}\to
\int_{\Sigma^{\delta}}\langle \Psi_{\infty},\Theta_\infty\rangle d\mu_h\eeq
for every $\de>0$ -- see also Lemma \ref{lemma:thick} in the appendix.

We obtain the first claim of the theorem passing to the limit $\de\to 0$ since the integrals in the above formula converge
uniformly to the corresponding integrals over the full surface as $\de\to0$ thanks to the estimate
$$\abs{\int_{M\setminus M_i^\delta}\langle \Psi_i,\Theta_i\rangle d\mu_{g_i}}\leq \norm{\Psi_i}_{L^1}\cdot \norm{\Theta_i}_{L^\infty(M\setminus M_i^\delta)}\leq C\cdot\delta^{-2}e^{-\pi/\delta}$$
resulting from Lemma \ref{Wi_structure}.

For the proof of the second statement we let $\{\Th_\infty^j\}$ be any unitary basis of $\Hol(\Si,h)$ and, using
Lemma \ref{Wi_structure}, choose for each $W_i$ a unitary basis $\{\Th_i^j\}$ such that $f_i^*\Th_i^j\to \Th_\infty^j$ smoothly locally for every $j\in\{1,\ldots 3(\gamma-1)-k\}$
as $i\to \infty$.
Then given any sequence of quadratic differentials $\Psi_i$ as in (ii), and abbreviating $\dlangle \Psi, \Th\drangle_{(M,g)}=\int_M\langle \Psi,\Th\rangle d\mu_g$, we find
\beqa f_i^*(P_{g_i}^{W_i}(\Psi_i))=& f_i^*\left(\sum_{j=1}^{3(\gamma-1)-k}\dlangle \Psi_i, \Th_i^j\drangle_{(M,g_i)} \cdot \Th_i^j\right)
=\sum_{j=1}^{3(\gamma-1)-k} \dlangle \Psi_i, \Th_i^j\drangle_{(M,g_i)} \cdot f_i^*\Th_i^j
\nn
\\
\to& \sum_{j=1}^{3(\gamma-1)-k}\dlangle \Psi_\infty, \Th_\infty^j\drangle_{(\Si,h)} \cdot \Th_\infty^j =P_h^{\Hol(\Si,h)}(\Psi_\infty),
\eeqa
smoothly locally, using the first part of the lemma.
The final claim in (ii) follows from Lemma \ref{Wi_structure} (iii).
\end{proof}

\section{Asymptotic convergence in the general degenerate case}
\label{mainthmpf}
Now we have developed enough theory for quadratic differentials in order to prove our main theorem.
\begin{proof}(Theorem \ref{thm:asymptotics}.)
Let $(u,g)$ be a global solution of \eqref{flow} as in Theorem \ref{thm:asymptotics}.
Recall that $(u,g)$ is smooth away from finitely many times and that the energy decays according to
\begin{equation}
\label{energy_decay}
\frac{dE}{dt}=-\int_M |\tau_g(u)|^2+\left(\frac{\eta}{4}\right)^2 |Re(P_g(\Phi(u,g)))|^2
\end{equation}
on any interval on which the solution is smooth. 
We can thus choose a sequence $t_i\to \infty$ such that
$$\norm{\tau_g(u)(t_i)}_{L^2(M,g(t_i))}\to 0 \text{ and  } \norm{P_{g}(\Phi(u,g))(t_i)}_{L^2(M,g(t_i))}\to 0 \text { as } i\to \infty.$$

Passing to a further subsequence
we obtain that the surfaces $(M,g(t_i))$ degenerate to a  hyperbolic punctured surface as described in Proposition \ref{pro1.1}
(modulo diffeomorphisms $f_i$).

Since we are dealing with a gradient flow for energy,
the energies of $u_i:=u(t_i)\circ f_i$ computed with respect to the varying metrics $G_i:=f_i^*g(t_i)$ are uniformly bounded, and indeed
$E(u_i,G_i)\leq E(u_0,g_0)$. Since
the metrics $G_i\to h$ converge locally uniformly, we also have a uniform upper bound for the energies of $u_i$
given by
$$\limsup_{i\to \infty} E(u_i,(K,h))\leq E(u_0,g_0),$$
valid for every compact subset $K\subset\subset \Si$.

Passing to a subsequence we may thus assume that $u_i$ converges weakly in $H^1_{loc}(\Si,h)$ to a limit map $u_\infty$ with finite energy,
and we claim that $u_\infty$ is both harmonic as well as weakly conformal.

The proof that $u_\infty$ is harmonic is very similar to the non-degenerate case \cite{R-T}. We let  $\eps_0=\eps_0(N)>0$ be such that the basic $\eps$-regularity estimate
\beq \label{est:stand-reg}\int_{D_r} \abs{\na^2 u}^2 \leq \frac{C}{r^2} E(u;D_{2r})+C\norm{\tau(u)}_{L^2(D_{2r})}^2\eeq
is valid for all maps into $N$ with energy less than $\eps_0$ on the Euclidean disc (cf. Lemma 3.3 of \cite{R-T}) and consider the finite set of concentration points
$$S:=\{p\in \Si: \limsup_{i\to \infty} E(u_i,(U,h))>\eps_0 \text{ for every neighbourhood }U \text{ of } p \}.$$
As in the non-degenerate case, from \eqref{est:stand-reg} and the convergence of metrics we obtain uniform $H^2$ bounds for $u_i$ on compact subsets of $\Si\setminus S$ which
allow us to extract a subsequence that converges weakly in $H^2_{loc}(\Si\setminus S)$ and strongly in $W^{1,p}_{loc}(\Si\setminus S)$, $1\leq p<\infty$, to a limit which must of course agree with $u_\infty$ where it is defined.
Since we chose the sequence of times $t_i$ in such a way that
$$\norm{\tau_{G_i}(u_i)}_{L^2(\Si,G_i)}\to 0$$
the limit $u_\infty$ must be harmonic, initially on $\Si\setminus S$ but then, by the removable singularity
theorem \cite{SU} and the finiteness of the energy,
on all of $\Si$, and indeed on the compactification $\overline \Si$ obtained from $(\Si,c)$ by filling in each of the $2k$ punctures. Note in particular, that $u_\infty$ is smooth throughout $\Si$, as is its extension to $\overline\Si$.

We conclude in particular that the Hopf-differential $\Phi(u_ \infty, h)$ is holomorphic. Furthermore, its $L^1$ norm is bounded by the total energy of $u_\infty$ on
$(\Si,h)$ and is thus finite, which means that $\Phi(u_\infty,h)\in \Hol(\Si,h)$.
In order to prove that $u_\infty$ is (weakly) conformal
(i.e. that $\Phi(u_\infty,h)\equiv 0$) it is thus enough to show that the projection $P_h^{\Hol(\Si,h)}(\Phi(u_\infty,h))$ vanishes.

But we know that $\|\Phi(u(t_i),g(t_i))\|_{L^1(M,g(t_i))}$ is bounded
and that also the $L^1$ norm of its antiholomorphic derivative
$\|\bar\partial \Phi(u(t_i),g(t_i))\|_{L^1(M,g(t_i))}$ is bounded.

Locally, we can thus apply the argument of the non-degenerate case: the Compactness Lemma 2.3 of \cite{R-T} combined with the convergence of complex structures (and thus the convergence
of \emph{isometric coordinate charts} as defined in the proof of Lemma \ref{lemma:holomorphic}) implies that 
$\Phi(u_i,G_i)$
converges to a limit $\Psi_\infty$ which we know is again a quadratic differential on $(\Si,h)$. Indeed, due to the strong $H^1$ convergence of $u_i$
on the complement of $S$, the limit $\Psi_\infty$ must agree with the Hopf-differential
$\Phi(u_\infty,h)$ of the limiting map.

Since the projections $P_{g(t_i)}^{W_i}$ onto the subspaces $W_i\subset \Hol(M,g(t_i))$ defined in Lemma \ref{Wi_structure} converge
to $P_h^{\Hol(\Si,h)}$ as described in Theorem \ref{neck_lemma} we conclude that
$$P_h^{\Hol(\Si,h)}(\Phi(u_\infty,h))=\lim_{i\to \infty}
f_i^*\left[P_{g(t_i)}^{W_i}(\Phi(u(t_i),g(t_i)))\right]= 0,$$
where the last equality is due to the projection of $\Phi(u(t_i),g(t_i))$ onto the full space $\Hol(M,g(t_i))$ converging to zero.

Thus $u_\infty$ is a weakly conformal harmonic map on $\Si$ and thus \cite{GOR}, on each of its connected components, either a branched minimal immersion or constant.
\end{proof}

\appendix

\section{Degenerating hyperbolic surfaces}
In this appendix we collect several important definitions and facts from the theory of hyperbolic surfaces that have been used throughout the main part of this paper, emphasising the geometric analysis aspects that we require.

\subsection{Deligne-Mumford compactness theorem}
\label{DMsubsection}

The classical Mumford Compactness Theorem \cite{mumford, tromba} tells us that if we have a sequence of closed hyperbolic Riemann surfaces $(M,g_i,c_i)$, and the length $\ell(g_i)$ of the shortest closed geodesic is bounded uniformly away from zero, then a subsequence
converges to a limiting hyperbolic surface
$(M,g_\infty,c_\infty)$ of the same topological type, in the sense that there exists a family of diffeomorphisms $f_i:M\to M$ such that
$$f_i^*g_i\to g_\infty \text{ and } f_i^*c_i\to c_\infty \text{ smoothly on } M.$$

Here, the convergence of a sequence of complex structures ${\mathfrak{c}}_i$ (here $f_i^*c_i$)
to a limit complex structure $\mathfrak{c}$ (here $c_\infty$) means that around each point in the underlying space (here $M$) there exists a neighbourhood $U$ on which there are complex coordinates $z_i$ (with respect to $\mathfrak{c}_i$) and
$z$ (with respect to $\mathfrak{c}$) such that $z_i \to z$ in $C^\infty$ on $U$.

We will primarily be interested in the general case that $\ell(g_i)$ has no uniform positive lower bound, in which case we will get a more general limit of the following form:

\begin{defn}
$(\Si,c)$ is called a Riemann surface of genus $\gamma\in\N_0$ and with $K\in\N_0$ punctures, if $\Sigma=\n\setminus\{p^1,...,p^K\}$, where $(\n, \overline{c})$ is a closed Riemann surface of genus $\gamma$, $\{p^1,...,p^K\} \subset \n$ and $c$ is the complex structure induced from $\overline{c}$. $(\Si,c)$ is said to be of general type, if $2\gamma+K>2$. By the uniformization theorem, any Riemann surface $(\Si,c)$ of general type can be equipped with a complete hyperbolic metric $h$ that is compatible with the complex structure $c$. We then call $(\Si,h)$ a hyperbolic punctured surface.
Throughout, we adopt the convention that $\n$ and $\Si$ may have more than one (but finitely many) components.
\end{defn}

The Deligne-Mumford Compactness Theorem in the following form explains how degeneration can occur in the genus $\ga\geq 2$ case when $\ell(g_i)$ can decay to zero. Although the area of each surface is fixed (by Gauss-Bonnet) it can stretch out and become noncompact in the limit.
\begin{prop} {\rm (Deligne-Mumford compactness \cite{DM,Hu}.)}  \label{pro1.1}
Let $(M,g_{i},c_{i})$ be a sequence of closed hyperbolic Riemann surfaces of genus $\gamma\geq2$ which degenerate in the sense that there is no uniform positive lower bound on $\ell(g_i)$. Then, after selection of a subsequence,
$(M,g_{i},c_{i})$ converges (in a sense to be clarified) to a hyperbolic punctured Riemann surface $(\Sigma,h,c)$, where $\Si$ arises as follows:
There exists $\mathscr{E}=\{\sigma^{j}, j=1,...,k\}$, a collection of $k$ pairwise disjoint, homotopically nontrivial, simple closed curves on $M$ so that if
$\widetilde{M}$ is the surface obtained from $M$ by pinching all curves $\sigma^{j}$ to points $q^{j}$, the surface $\Si$
is defined to be $\widetilde{M}\setminus \cup_{j=1}^{k}q^{j}$.

The convergence above is to be understood as follows:
For each $i$ there exists a collection $\mathscr{E}_{i}=\{\sigma^{j}_{i}, j=1,...,k\}$ of
pairwise disjoint simple closed geodesics on $(M,g_{i},c_{i})$
with each $\si_i^j$ homotopic to $\si^j$,
and a continuous map $\tau_{i}: M \rightarrow \widetilde{M}$ with $\tau_{i}(\sigma_{i}^{j})=q^{j}$, $j\in \{1,...,k\}$ such that:

\begin{itemize}
\item[(i)] For each $j\in \{1,...,k\}$, the lengths $\ell(\sigma_{i}^{j})=:\ell_{i}^{j} \rightarrow 0$ as $i \rightarrow \infty$ .

\item[(ii)] For each $i$, $\tau_{i}: M\setminus \cup_{j=1}^k\sigma_{i}^{j} \rightarrow \Sigma$ is a diffeomorphism and its inverse is denoted by $f_i: \Sigma \rightarrow M\setminus \cup_{j=1}^k\sigma_{i}^{j} $.

\item[(iii)] $(f_i)^{*}g_{i} \rightarrow h $ in $C_{loc}^{\infty}$ on $\Sigma$.

\item[(iv)] $(f_i)^{*}c_{i} \rightarrow c $ in $C_{loc}^{\infty}$ on $\Sigma$.
\end{itemize}
\end{prop}

By hyperbolic surface theory, the number of simple closed geodesics of length $< {2\arsinh(1)}$ for a closed hyperbolic surface of genus $\gamma\geq 2$ is bounded by $3\gamma-3$ (cf. \cite{Hu}). Therefore, in Proposition
\ref{pro1.1}, we have $1 \leq k \leq 3\gamma-3$.

More generally, we have:

\begin{prop}(\cite{Hu} Lemma 4.1.) \label{Thick-thin-1}    
Let $(\Sigma,h)$ be a  hyperbolic punctured surface of genus $\gamma$ and with $K$ punctures. Then the simple closed geodesics in $\Sigma$ of lengths smaller than ${2\arsinh(1)}$ are pairwise disjoint. In particular, there are only finitely many of them and their number is bounded by $3\gamma-3+K$.
\end{prop}

\subsection{Description of the thin parts of the surface: Collars and Cusps}
\label{thinsect}

Let $(\m,h)$ be any smooth Riemannian manifold. We denote by
$\inj(p)=\inj_{(\m,h)}(p)$
the injectivity radius of $(\m,h)$ at $p\in\m$ and by $\m^\de$ the $\de$-thick part
$$\m^\de:=\{p\in \m:\, \inj(p)\geq \de\},\quad \de>0.$$
The $\de$-thin part of $\m$, sometimes denoted by $\de\text{-thin}(\m)$,
is then the open set $\m\setminus \m^\de$ of points with injectivity radius strictly less than $\de$.

One fundamental fact of hyperbolic surface theory
-- see Proposition \ref{Thick-thin-2} below --
is that for any $0 <\delta < {\rm arsinh(1)}$, the $\delta$-thin part
of a hyperbolic surface is given by a finite union of hyperbolic cylinders of finite length around closed geodesics,
and of cylinders of infinite length that we call standard cusps.
The regions near simple closed geodesics are described by the Collar Lemma:

\begin{lemma} (Keen-Randol Collar Lemma, {\rm \cite{R}}.)
\label{Collar-lemma}
Let $(M, g)$ be a closed hyperbolic surface and let $\sigma$
be a simple closed geodesic of length $\ell$. Then there is a neighbourhood $U$ around $\sigma$, a so called collar, which is isometric to the cylinder
$$ \cc(\ell)=(-X(\ell), X(\ell)) \times S^1$$ equipped with the hyperbolic metric
$\rho^2(ds^2+d\theta^2)$, where $$X(\ell)= \frac{2\pi}{\ell}\left(\frac{\pi}{2}-\arctan\left(\sinh\left(\frac{\ell}{2}\right)\right) \right),  \qquad \rho=\frac{\ell}{2\pi \cos(\frac{\ell s}{2\pi})}.$$
The geodesic $\sigma$ corresponds to the circle 
$\{0\}\times S^1\subset \cc(\ell) $.
\end{lemma}

Owing to this explicit description of the metric in these collar regions, we can read off the $\de\thin$ part:

\begin{lemma} (\cite{Hu, Zh}.)
\label{lem2.3} Let $(\cc(\ell),\rho^2(ds^2+d\th^2))$, $0<\ell\leq 2 \arsinh(1)$ be a collar region of a hyperbolic surface $(M,g)$ as described in the Collar Lemma \ref{Collar-lemma}. Then the
injectivity radius of $(M,g)$ at points in the collar is characterised by
\bee  \sinh({\rm inj}(s,\theta))\cdot \cos\left(\frac{\ell s}{2\pi}\right) =  \sinh\left(\frac{\ell}{2}\right). \nn  \eee
In particular, given any such $\ell$ and any $0<\de<\arsinh(1)$ 
the $\de$-thin part of the collar  
is given by the subcylinder
\beq \cc(\ell,\de):=(-X_\de(\ell), X_\de(\ell)) \times S^1 \subseteq \cc(\ell),    \eeq
where $X_\de(\ell)=  \frac{2\pi}{\ell}\left(\frac{\pi}{2}-\arcsin \left(\frac{\sinh(\frac{\ell}{2})}{\sinh \delta}\right) \right)$ for $\de\geq \ell/
2$, respectively zero for smaller values of $\delta$, and its area is bounded by
$${\rm Area}(\cc(\ell,\de))\leq \frac{2\ell}{\sinh (\frac{\ell}{2})}\cdot \sinh \delta \leq C\cdot \delta$$
for a uniform constant $C$ independent of $0<\ell\leq 2\arsinh(1)$ and $0<\de<\arsinh(1)$.
\end{lemma}

\begin{proof}
The formula for the injectivity radius and thus for the $\de$-thin part of the collar follows from arguments of hyperbolic geometry, see \cite{Hu}.
The area of $\cc(\ell,\de)$ is then given by
\beqa  {\rm Area} (\cc(\ell,\delta)) &= \int_{\cc(\ell,\delta)} \rho^2 ds d \theta =\frac{\ell^2}{2\pi} \int_{-X_\de}^{X_\de}\cos^{-2}\left(\frac{\ell s}{2\pi}\right) ds\\
&=2  \ell\tan\left(\frac{\ell X_\de}{2\pi}\right)\leq 2 \frac{\ell\cdot \sinh(\delta)}{{\rm sinh} (\frac{\ell}{2})}\leq C\cdot \de
\eeqa
for $0<\de<\arsinh(1)$ as claimed.
\end{proof}

The Collar Lemma also applies to compact subsets of  hyperbolic punctured surfaces $(\Si,h)$ and thus gives that for
$0 <\delta < {\rm ar sinh(1)}$ the $\de$-thin part of $(\Sigma,h)$ consists of (a possibly empty set of)
collars as well as the $\de$-thin parts of the surface contained in neighbourhoods of the punctures.
Near a puncture, $(\Si,h)$ has the form of a so-called standard cusp, leading to the following general description of the thin part of a hyperbolic punctured surface.

\begin{prop} (\cite{Hu} Proposition IV.4.2.) \label{Thick-thin-2} Let $(\Sigma,h)$ be a hyperbolic punctured surface with punctures $\{p^1,...,p^K\}$.
Then the ${\rm ar sinh(1)}$-thin part of $(\Sigma,h)$ is given as 
the union of  mutually disjoint sets consisting of
\begin{enumerate}
\item the ${\rm ar sinh(1)}$-thin parts of the collar neighbourhoods 
around simple closed geodesics $\si$ of length $\ell=\ell(\si) <2 {\rm ar sinh(1)}$, and 
\item neighbourhoods $U(p^j)$ around each of the punctures, $j=1\ldots K$, which are all isometric to a standard cusp,
i.e. to the infinitely long half-cylinder $(\pi,\infty)\times S^1$ equipped with the metric $\frac{1}{s^2}(ds^2+d\theta^2)$  or equivalently to the punctured open disc
$D_{e^{-\pi}}(0)\setminus\{0\}$ equipped with the metric $\frac1{|z|^2\cdot(\log(\abs{z}))^2} dz^2$.
\end{enumerate}
\end{prop}

For a degenerating sequence of hyperbolic surfaces,
we now conclude:
\begin{lemma} \label{lemma:thick}
Suppose $(M,g_i)$ converges to a  hyperbolic punctured surface $(\Sigma,h)$ as in Proposition \ref{pro1.1}.
Then the following claims are true for any number $0<\delta<\arsinh(1)$:
\begin{enumerate}
 \item The $\de$-thick parts $\Si_i^\de$ of the surfaces $(\Si,f_i^*g_i)$ converge to the compact set $\Si^\de$, the $\de$-thick part of the limit surface,
both in the sense of Hausdorff distance on $(\Sigma,h)$ as well as in the sense that the measure of the symmetric difference $\Si_i^\de\Delta\Si^\de$ converges to zero.
(In particular, all the sets are contained in a uniform compact set.)
\item Given any sequence of functions $\varphi_i:\Si\to\R$ with $\norm{\varphi_i}_{L^1(\Si,f_i^*g_i)}$ bounded and converging
$$\varphi_i\to \varphi\quad \text{ in } L_{loc}^1(\Sigma,h),$$
the corresponding integrals over the $\de$-thick parts of the surfaces converge
$$\int_{\Si_i^\de}\varphi_i d\mu_{f_i^*g_i}\to \int_{\Si^\de}\varphi d\mu_h.$$
\end{enumerate}
\end{lemma}

\begin{proof} 
Let $0<\de<\arsinh(1)$ be fixed.
We first recall that the limiting surface can be decomposed into a compact set $K_0$ as well as the neighbourhoods $U(p^j)$, $j\in\{1\ldots 2k\}$ of the punctures which are isometric to the cylinders
$(\pi,\infty)\times S^1$ equipped with the metric
$s^{-2}(ds^2+d\theta^2)$. Since the $\de$-thick part of such a cusp is also compact, the whole $\de$-thick part $\Si^\de$ of the limiting surface $(\Si,h)$ is compact
and thus the metrics $f_i^*g_i$ converge
smoothly on $\Si^\de$. In particular, given any $\eps>0$ and $\de>0$ we find that
\beq\label{est:inj-rad} \sup_{z\in \Si^\de} \abs{\inj_{f_{i}^*g_i}(z)-\inj_{h}(z)}<\eps\eeq
for $i$ sufficiently large, say $i\geq i_0(\eps,\de)$, so we conclude that $\Si^\de\subset \Si_i^{\de-\eps}$.

By the same argument, for $K$ any fixed compact subset of $\Si$, we conclude that the points in $K\setminus \Si^{\de}$ are eventually in the $\de+\eps$ thin part of the
degenerating surfaces $(\Si,f_i^*g_i)$. In order to obtain the inclusion
\beq \label{est:inclusion}\Si_i^{\de+\ep}\subset \Si^\de\eeq
for $i$ sufficiently large, we thus need to prove that there is a \textit{uniform} compact subset $K(\de+\eps)$ of $\Si$ which contains
the $\tilde\de=\de+\eps$ thick parts of the degenerating surfaces $(\Si,f_i^*g_i)$
for all $i$ sufficiently large. 
Using the decomposition of $\Si$ into a compact set $K_0$ and the neighbourhoods of the punctures $U(p^j)$ it is enough to establish this claim for points contained in
$U(p^j)$, $j\in\{1\ldots 2k\}$ and we analyse the $\tilde\de\thin$ parts

$$\mathcal{U}_i^j(\tilde \de):=\tilde \de\text{-thin}(U(p^j),f_i^*g_i)=\{p\in U(p^j): \inj_{f_i^*g_i}(p)<\tilde \de\},\quad 0<\tilde\de<\arsinh(1)$$
of these cylindrical neighbourhoods of the punctures.

We recall that the maps $\tau_i:M\to \widetilde M$ described in Proposition 
\ref{pro1.1}
are continuous
and map $\si_i^j$ to the point $q^j$ in $\widetilde M$, which in turn corresponds to a pair of punctures in $\Si$.

Thus for $i$ sufficiently large, $\mathcal{U}_i^j(\tilde\de)$ must be (topologically) a cylinder which contains in particular all points
close to the puncture, i.e.~using the description of $U(p^j)$ in Proposition \ref{Thick-thin-2}, in particular all points with large (a priori depending on $i$) $s$-coordinate.

But on the other hand, by what we already proved the
compact cylinder
$\{p\in U(p^j): \inj_{h}(p)\in[\tilde \de/4,\tilde \de/2]\}$
is eventually contained in $\mathcal{U}_i^j(\tilde \de)$ so for $i$ sufficiently large
the whole set $U(p^j)\setminus \Si_{\tilde \de/2}$ is contained in
$\mathcal{U}_i^j(\tilde\de)$. Equivalently we obtain that $\tilde \de\thick(U(p^j))=U(p^j)\setminus \mathcal{U}_i^j(\tilde \de)$ is contained in $\Si^{\tilde\de/2}$, for $i$ large, 
and thus that the sets $\Si_i^{\tilde\de}$ are contained in a uniform compact subset of $(\Si,h)$.

All in all we thus conclude that for any $\de>0$, any $\eps>0$ and for $i$ large enough
\beq \label{est:Si^delta}\Si_i^{\de+\eps}\subset \Si^\de\subset\Si_i^{\de-\eps} .\eeq

The set $\Si^{\de-\eps}\setminus \Si^{\de+\eps}$, $0<\de+\eps<\arsinh(1)$, is now given by a union of subcylinders $[X_{\de-\eps},X_{\de+\eps}]\times S^1$ of collar respectively
puncture regions that are explicitly described by Lemma \ref{Collar-lemma} and Proposition \ref{Thick-thin-2} and it is easy to see that the 
distance between the two boundary curves of each such cylinder converges
to zero as $\eps\to0$. Combined with \eqref{est:Si^delta} this in particular implies the convergence of the sets $\Si_i^\de$ to $\Si^\de$ as described in the first statement of the 
lemma.

To obtain the second claim of Lemma \ref{lemma:thick} we now exploit that local convergence of functions and metrics implies uniform convergence
on the compact set $\Si^{\de/2}$, $\de>0$, which contains $\Si_i^\de$ for $i$ sufficiently large.
Combining the convergence of the integrals on the fixed set $\Si^\de$
$$\int_{\Si^\de}\varphi_i d\mu_{f_i^*g_i}\to \int_{\Si^\de}\varphi d\mu_h$$
with the fact that the symmetric difference $\Si^\de\Delta\Si_i^\de$ is contained in the compact set $\Si^{\de/2}$ for $i$ large and that its measure converges to zero
implies that
$$\limsup_{i\to\infty}\int_{\Si^\de\Delta\Si_i^\de}\abs{\varphi_i} d\mu_{f_i^*g_i}\leq \limsup_{i\to\infty}\int_{\Si^\de\Delta \Si_i^\de}\abs{\varphi} d\mu_{h}= 0,$$
so that we obtain the second claim of the lemma.
\end{proof}

\subsection{Estimates for holomorphic quadratic differentials on hyperbolic surfaces}
We finally collect a few useful properties of holomorphic quadratic differentials.
We first remark that the $L^p$ norms over the thick part of the surface are controlled by the $L^1$ norm
\begin{lemma}\label{lemma:Lp}
For any $\de>0$ and any closed surface $M$ there exists a constant $C<\infty$ depending only on $\de$ and the genus of $M$ such that for every hyperbolic metric $g$ on $M$
$$\norm{\Th}_{L^p(M^\de,g)}\leq C\cdot \norm{\Th}_{L^1(M,g)}, \quad\text{ for all } 1\leq p\leq \infty, \quad\Th\in\Hol(M,g), $$
where $M^\de:=\de\thick(M,g)$.
\end{lemma}

Furthermore we clarify how holomorphicity leads to derivative estimates in terms of the $L^1$ norm, in the presence of degeneration.

\begin{lemma}\label{lemma:holomorphic}
Let $(M,g_i)$ be any sequence of hyperbolic surfaces that converges to a (possibly punctured) limiting surface $(\Sigma,h)$ as described in Proposition \ref{pro1.1}. Then holomorphic quadratic differentials are uniformly controlled on the thick part of the surfaces in the following sense:
For any $\de>0$ and any $m\in \N_0$, there exists a
uniform  ($i$-independent)
constant $C<\infty$ such that for all holomorphic quadratic differentials
$\Th_i\in \Hol (M,g_i)$
$$\norm{f_i^*\Th_i}_{C^m(\Si_i^\de)}\leq C\cdot \norm{\Th_i}_{L^1(M,g_i)}, \quad i\in\N.   $$
The same bound is valid also on the limiting surface
$$\norm{\Th_\infty}_{C^m(\Si^\de)}\leq C\cdot\norm{\Th_\infty}_{L^1(\Si,h)}$$
for all $\Th_\infty\in \Hol(\Si,h)$.
\end{lemma}
Here and in the following we compute the $C^m$ norm with respect to a fixed set of coordinate charts of $\Si$.

\begin{rmk}\label{rmk:Linfty-limit}
Thanks to the rapid decay on collars of elements of $W_i$ as well as the convergence of these spaces to $\Hol(\Si,h)$ as discussed in Theorem \ref{neck_lemma},
under the assumptions of the previous lemma we could also obtain
an estimate of the form
\beq
\norm{\Th}_{L^p(\Si,h)}\leq C\cdot \norm{\Th}_{L^1(\Si,h)}, \text{ for all } 1\leq p\leq \infty, \quad\Th\in\Hol(\Si,h),
\eeq
giving a bound on the \textit{global} $L^p$ norm of elements of $\Hol(\Si,h)$. Here the constant $C<\infty$ depends on $(\Si,h)$.
\end{rmk}

We first give a very short proof of Lemma \ref{lemma:Lp}.
\begin{proof}(Lemma \ref{lemma:Lp}.)
Let $(M,g)$ be a hyperbolic surface, let $\Th\in \Hol(M,g)$ and let $\de>0$. We first remark that with the area of $(M,g)$ determined by its genus,
it is sufficient to bound the $L^\infty$ norm of $\Th$. Given any point $z_0\in M^\de$ we choose a coordinate chart
$$\phi:B_{g}(z_0,\de)\to(B_{g_H}(0,\de),g_H)$$
which is an isometry from the $\de$-ball around $z_0$ in $(M,g)$ to the ball of radius $\de$ in the Poincar\'e hyperbolic disc. In this coordinate chart $\Th$ is given as
$\th dz^2$ for a holomorphic function $\th:B_{g_H}(0,\de)\to \C$. Standard estimates from complex analysis imply that the $L^1$ norm of the function $\th$ bounds its $L^\infty$ norm on a slightly
smaller ball, so that
\beqas
\norm{\Th}_{L^\infty(B_{g}(z_0,\de/2),g)}&\leq \norm{\th}_{L^{\infty}(B_{g_H}(0,\de/2))}\cdot \norm{dz^2}_{L^{\infty}(B_{g_H}(0,\de/2),g_H)}\\
&\leq C_\de\norm{\th}_{L^1(B_{g_H}(0,\de))}
\leq C_\de\norm{\Th}_{L^1(B_{g}(z_0,\de),g)}\eeqas
for a constant $C_\de$ depending only on $\de$. 
\end{proof}

\begin{proof}(Lemma \ref{lemma:holomorphic}.)
For the given $\de>0$, choose a finite cover of $\Si^\de$ consisting of balls $B_{h}(z^j,\de/4)\subset (\Si,h)$ with centres $z^j\in\Si^\de$.
Then for $i$ large enough, say $i\geq i_0$, also the $\de$-thick set $\Si_i^\de$ of $(\Si,f_i^*g_i)$ is covered by these balls and we may furthermore assume that
$\inj_{f_i^*g_i}(z^j)\geq \frac\de 2$ for each $j$.

Since the complex structures converge, there is a sequence of atlases which consist of coordinate charts that can
be viewed as isometries
$$\phi_i^j:B_{f_i^*g_i}(z^j,\de/2)\to(B_{g_H}(0,\de/2),g_H)$$
from the balls $B_{f_i^*g_i}(z^j,\de/2)$ of radius $\de/2$ in $(\Si,f_i^*g_i)$ to the fixed ball $B_{g_H}(0,\de/2)$ of radius $\de/2$ in the Poincar\'e hyperbolic disc
and such that the maps $\phi_i^j$ converge smoothly locally to an isometry $\phi_\infty^j$
from $B_{h}(z^j,\de/2)\subset (\Si,h)$ to
$(B_{g_H}(0,\de/2),g_H)$.

Working on the fixed domain $B_{g_H}(0,\de/2)$, standard complex analysis gives uniform $C^m$ bounds of
$$\norm{\th_i^j}_{C^m(B_{g_H}(0,\de/4))}\leq C_\de\cdot \norm{\th_i^j}_{L^1(B_{g_H}(0,\de/2))}\leq
C_\de\norm{\Th_i}_{L^1(M,g_i)}$$
for the holomorphic functions $\th_i^j$ that represent $f_i^*\Th_i$ with respect to the coordinate charts $\phi_i^j$.

Combined with the convergence of the charts $\phi_i^j$, these estimates translate to
uniform $C^m$ bounds on $f_i^*\Th_i$
$$\norm{f_i^*\Th_i}_{C^m(\Si_i^\de)}\leq C\cdot \sup_{j}\norm{\th_i^j}_{C^m(B_{g_H}(0,\de/4))}\leq C_\de\norm{\Th_i}_{L^1(M,g_i)}$$
with respect to the (fixed) isometric coordinate charts $\phi_\infty^j$ of $(\Si,h)$.

The estimate for holomorphic quadratic differentials on the limiting surface are an immediate consequence of the
$C^m$ estimates for holomorphic functions since we can work directly with respect to the complex coordinate charts $\phi_\infty^j$ of $(\Si,h)$ from above.
\end{proof}

We finally remark that the space $\Hol(\Si,h)$ of holomorphic quadratic differentials with finite $L^1$ norm can be equivalently characterised
as follows.

\begin{lemma} \label{lemma:poles-norms}
Let $(\Sigma, h, c)$ be a  hyperbolic punctured surface. Then for any holomorphic quadratic differential $\Phi$ on $(\Si,h)$ the following statements are equivalent:
\begin{itemize}
\item[(i)]   $\Phi\in \Hol(\Si,h)$, that is $\norm{\Phi}_{L^1(\Si,h)}$ is finite.
\item[(ii)] $\Phi$ is bounded (with respect to the hyperbolic metric $h$).
\item[(iii)] At each of the punctures of $(\Si,c)$ the differential $\Phi$ has at worst a simple pole.
\end{itemize}
\end{lemma}

The last statement implies in particular that an element of $\Hol(\Si,h)$ cannot have an essential singularity at a puncture, so we could equivalently
say that elements of $\Hol(\Si,h)$ are
meromorphic with poles of order no more than $1$.
\begin{proof}
Let $\Phi$ be any holomorphic quadratic differential on $(\Si,h)$. We remark that according to Lemma \ref{lemma:holomorphic} it is enough to
consider $\Phi$ on neighbourhoods $U(p^j)$ of the
punctures as described in Proposition \ref{Thick-thin-2}, that is on punctured discs $\hat D$ equipped with the hyperbolic metric
$(\abs{z}\cdot \log \abs{z})^{-2} \abs{dz}^2$.

We remark that a holomorphic function on a punctured disc $\hat D$
with finite $L^1$ norm can neither have a pole of order more than one, nor an essential singularity.
One way of seeing that is to appeal to
the subharmonicity of the absolute value of the holomorphic function, applying it on discs of radius $\abs{z}$ around $z$.

Thus $\norm{\Phi}_{L^1(U(p^j),h)}=2\int_{\hat D}\abs{\phi} dxdy$ 
is finite on each of the puncture regions $U(p^j)$
if and only if (iii) holds and we conclude that (i) and (iii) are equivalent.

Finally, we recall that on such a neighbourhood $U(p^j)$,
$$\abs{\Phi}(z)=\abs{\phi}\abs{dz^2}
=2\abs{\phi}\abs{z}^2(\log \abs{z})^2,$$
with our normalisation,
and thus that a holomorphic quadratic differential with a simple pole is bounded (with respect to $h$) so that (iii) implies (ii) which trivially implies (i).
\end{proof}

With $\Hol(\Si,h)$ characterised as the space of meromorphic quadratic differentials (with poles of order no more than $1$) its dimension is now described by the
Riemann-Roch Theorem

\begin{lemma} \cite{Ga} \label{lem2.6}
Let $(\Si,h)$ be a complete (not necessarily connected) hyperbolic surface with $K\in \N_0$ punctures.
Then the dimension of the space of integrable holomorphic quadratic differentials is 
\beq \label{dim-formula} {\rm dim}_{\mathbb{C}}\Hol(\Si, h) = \sum_i 3(\gamma_i-1) + K,  \nn
\eeq
where $\gamma_i$ is the genus of the $i$-th connected component of the compactification $\overline\Si$ of $\Si$ obtained by filling in the punctures.
\end{lemma}

\begin{rmk}\label{rmk:dim}
For a degenerating sequence of hyperbolic surfaces the (complex) dimension of the spaces $\Hol$ thus reduces in the limit $i\to \infty$ by exactly the number of collapsing collars.
Indeed, collapsing a closed, non-homotopically trivial curve $\sigma$ on a connected surface $\Sigma$ to a point and removing this point increases the number of punctures by two.
If $\sigma$ is a separating curve, this furthermore splits the surface $\Si$ into two parts of genus $\gamma_1+\gamma_2=\gamma$.
If $\sigma$ is not separating, then the resulting surface will have genus $\tilde \gamma=\gamma-1$.
In both cases the dimension of the space of integrable holomorphic quadratic differentials decreases by exactly one and the claim follows repeating the argument for all collapsing geodesics and
connected components.
\end{rmk}

M.~Rupflin:\\
{\sc Max-Planck-Institut f\"ur Gravitationsphysik, Am M\"uhlenberg 1, 14476 Potsdam, Germany}

P.M.~Topping and M.~Zhu:\\
{\sc Mathematics Institute, University of Warwick, Coventry,
CV4 7AL, UK}

\end{document}